\documentclass[11pt]{article}
\usepackage{amsmath,wasysym}
\usepackage{amsthm}
\usepackage{amssymb,amsfonts}
\usepackage{hyperref}
\usepackage{latexsym}
\usepackage{todonotes}
\usepackage{nicefrac}
\usepackage{epsfig}
\usepackage{stmaryrd}
\usepackage{setspace}
\usepackage{tikz-cd}
\usepackage{enumerate}
\usepackage[all]{xypic}
\usepackage{bbm,ifpdf,tikz,mathtools}
\usepackage{tikz-3dplot}

\usetikzlibrary{decorations.pathreplacing} 

\ifpdf
\usepackage{pdfsync}
\fi
\usetikzlibrary{positioning,matrix,arrows,decorations.pathmorphing}
\usepackage{enumitem}

\usetikzlibrary{positioning,arrows}
\usetikzlibrary{decorations.markings}
\usepackage{caption, subcaption}
\usetikzlibrary{arrows.meta,
                bbox}

\newtheorem{theorem}{Theorem}[section]

\newtheorem{lemma}[theorem]{Lemma}
\newtheorem{proposition}[theorem]{Proposition}
\newtheorem{corollary}[theorem]{Corollary}

\newtheorem*{main-result}{Main-Result}
{
\theoremstyle{definition}

\newtheorem{example}[theorem]{Example}

\newtheorem{remark}[theorem]{Remark}

}

\newcommand{\id}{\operatorname{id}}

\renewcommand{\dim}{\operatorname{dim}}

\newcommand{\cK}{\mathcal{K}}
\newcommand{\bigmid}{\;\Big{|}\;}

\newcommand{\cP}{\mathcal{P}}

\newcommand{\cJ}{\mathcal{J}}
\renewcommand{\cR}{\mathcal{R}}

\newcommand{\C}{{\mathbb{C}}}

\newcommand{\N}{{\mathbb{N}}}
\newcommand{\F}{{\mathbb{F}}}
\newcommand{\Q}{{\mathbb{Q}}}
\newcommand{\R}{{\mathbb{R}}}

\newcommand{\Hom}{\operatorname{Hom}}

\newcommand{\la}{\lambda}
\renewcommand{\a}{\alpha}

\newcommand{\becircled}{\mathaccent "7017}

\newcommand{\A}{\mathbb{A}}

\newcommand{\PA}{\P_{\!\cA}}

\newcommand{\cA}{\mathcal{A}}

\newcommand{\cE}{\mathcal{E}}
\renewcommand{\cD}{\mathcal{D}}
\newcommand{\cZ}{\mathcal{Z}}
\newcommand{\cF}{\mathcal{F}}
\newcommand{\cO}{\mathcal{O}}

\newcommand{\cQ}{\mathcal{Q}}
\newcommand{\cG}{\mathcal{G}}
\renewcommand{\cL}{\mathcal{L}}

\renewcommand{\and}{\qquad\text{and}\qquad}

\newcommand{\Aut}{\operatorname{Aut}}

\newcommand{\Conf}{\operatorname{Conf}}

\newcommand{\Sym}{\operatorname{Sym}}
\newcommand{\OT}{\operatorname{OT}}
\newcommand{\SR}{\operatorname{SR}}

\newcommand{\OTbar}{\overline{\operatorname{OT}}}

\newcommand{\SRbar}{\overline{\operatorname{SR}}}

\renewcommand{\P}{\mathbb{P}}
\newcommand{\Gm}{\mathbb{G}_m}

\newcommand{\Supp}{\operatorname{Supp}}
\newcommand{\beq}{\begin{eqnarray*}}
\newcommand{\eeq}{\end{eqnarray*}}
\newcommand{\rk}{\operatorname{rk}}
\newcommand{\crk}{\operatorname{crk}}
\newcommand{\fS}{\mathfrak{S}}

\usepackage{todonotes}

\begin{document}
\spacing{1.2}
\noindent{\Large\bf The geometry of zonotopal algebras II:\\  Orlik--Terao algebras and Schubert varieties
}\\

\noindent{\bf Colin Crowley}\footnote{Supported by NSF grant DMS-2039316.}\\
Department of Mathematics, University of Oregon, Eugene, OR 97403\\
and School of Mathematics, Institute for Advanced Study, Princeton, NJ 08540
\vspace{.1in}

\noindent{\bf Nicholas Proudfoot}\footnote{Supported by NSF grants DMS-1954050, DMS-2053243, and DMS-2344861.}\\
Department of Mathematics, University of Oregon, Eugene, OR 97403\\

{\small
\begin{quote}
\noindent {\em Abstract.} 
Zonotopal algebras, introduced by Postnikov--Shapiro--Shapiro, Ardila--Postnikov, and Holtz--Ron, 
show up in many different contexts, including approximation theory, representation theory, Donaldson--Thomas theory,
and hypertoric geometry.  
In the first half of this paper, we construct a perfect pairing between the internal zonotopal algebra of a linear space
and the reduced Orlik--Terao algebra of the Gale dual linear space.  
As an application, we prove a conjecture of Moseley--Proudfoot--Young that
relates the reduced Orlik--Terao algebra of a graph to the cohomology of a certain configuration space.  
In the second half of the paper, we interpret
the Macaulay inverse system of a zonotopal algebra as the space of sections of a sheaf on the Schubert variety of a linear space.
As an application of this, we use an equivariant resolution of the structure sheaf of the Schubert variety inside of a product of projective lines
to produce an exact sequence relating internal and external zonotopal algebras.
\end{quote} }

\section{Introduction}
Zonotopal algebras come in three flavors:  central, internal, and external.  Central zonotopal algebras and their duals play an important role in approximation theory,
dating back to the work of de Boor and H\"ollig \cite{dBH}; they have also arisen in the contexts of hypertoric geometry \cite{HS} and Donaldson--Thomas theory \cite{RRT}.  
Internal zonotopal algebras also have applications to approximation theory \cite{Lenz}, and to graphical configuration spaces \cite{CDBHP}.
External zonotopal algebras were introduced by Postnikov, Shapiro, and Shapiro in the context of differential forms on flag manifolds \cite{PSS}.
These three types of algebras were combined and regarded as parts of a single theory by Ardila and Postnikov \cite{ArPo}, and independently by Holtz and Ron \cite{Zonotopal}, who gave them their names.

Another important algebra associated with the same data is the Orlik--Terao algebra \cite{OT-OT}.
This algebra has connections to linear programming \cite{SSV}, linear codes \cite{Tohaneanu-Hamming}, Kazhdan--Lusztig theory of matroids \cite{EPW},
and the theory of resonance varieties \cite{Schenck-OT}; it also plays a role in the categorification of the Tutte polynomial in \cite{Berget}.  Our first main result is that the internal zonotopal algebra of a linear space 
is canonically isomorphic to the dual of the reduced Orlik--Terao algebra of the Gale dual linear space (Theorem \ref{internal}).
We use this to prove a conjecture of Moseley, Proudfoot, and Young \cite{MPY} relating the Orlik--Terao algebra 
of a graphical arrangement with the cohomology ring of a certain
configuration space (Corollary \ref{graph}), which has previously been proved only for the complete graph \cite{Pagaria}.

The Schubert variety of a linear space was introduced by Ardila and Boocher \cite{ArBoo},
and was used to prove non-negativity of the Kazhdan--Lusztig polynomial and $Z$-polynomial of a realizable matroid
\cite{EPW,PXY} and the Dowling--Wilson conjecture for realizable matroids \cite{HW}.  When the linear space arises from the Coxeter arrangement of a Lie algebra as in Remark \ref{hyperplanes}, 
the fundamental group of the real locus of the Schubert variety coincides with the pure virtual Weyl group
of the Lie algebra \cite{IKLPR,Jiang}.  We interpret the Macaulay inverse systems of all of the zonotopal algebras
as groups of sections of certain sheaves on the Schubert variety (Theorem \ref{Schubert}), which 
naturally appear in the minimal free resolution of the structure sheaf of the Schubert variety inside of a product of projective lines (Theorem \ref{resolution}).
This allows us to construct an exact sequence relating internal and external zonotopal algebras (Corollary \ref{cohomology}), 
equivariant with respect to any automorphisms of the arrangement, categorifying
a numerical formula for the number of spanning sets of a matroid (Remark \ref{spanning}).

\subsection{Background}\label{sec:background}
The basic data structure with which we will work is that of a triple $\cA = (E, V, L)$, where $E$ is a finite set, $V = \bigoplus_{e\in E} V_e$ is a vector space over a field $\F$
of characteristic zero
equipped with a decomposition into lines indexed by $E$, and $L\subset V$ is a linear subspace.  We will refer to such a triple as a {\bf linear space}.
We define the {\bf Gale dual} $\cA^! := (E, V^*, L^\perp)$.  The linear space $\cA$ has an associated matroid $M_\cA$ of rank equal to the dimension of $L$, and $M_{\cA^!}$ is the dual of $M_\cA$.
An automorphism of $\cA$ is by definition a linear automorphism of $V$ that preserves the decomposition into lines (possibly permuting them) and takes $L$ to itself.
Note that the automorphism group $\Aut(\cA)$ contains a distinguished copy of the multiplicative group $\Gm$, acting by scalar multiplication.  There is a canonical isomorphism
$\Aut(\cA)\cong\Aut(\cA^!)$ that restricts to inversion on $\Gm$.

\begin{remark}\label{hyperplanes}
Suppose that we are given a vector space $L$ and a multiset $\{H_e\mid e\in E\}$ of hyperplanes indexed by $E$ that intersect only in the origin.  Then we may define a linear
space $\cA$ by putting $V_e := L/H_e$.  An automorphism of $\cA$ consists of a pair $(\sigma,\tau)$, where $\sigma$ is a linear automorphism of $L$ and $\tau$ is a permutation of $E$ such
that $\sigma(H_e) = H_{\tau(e)}$ for all $e\in E$.  An arbitrary linear space $\cA$ arises via this construction if and only if the projection $L\to V_e$ is surjective for all $e$, or equivalently
if the matroid $M_\cA$ is loopless.  The main reason that we choose to work with linear spaces rather than hyperplane arrangements is that looplessness is not preserved by duality.
\end{remark}

For any element $\a\in L$, let $m(\a)$ be the number of coordinates $e\in E$ such that the projection $\a_e\in V_e$ is nonzero, and let
$\rho$ be the minimum value of $m(\a)$ for $0\neq \a\in L$.  For any integer $k\geq -(\rho+1)$, Ardila and Postnikov \cite{ArPo,ArPo-correction} define the ring 
$$R_{\cA,k} := \Sym L\Big{/} \left\langle \a^{m(\a)+k+1}\mid 0\neq \a\in L\right\rangle.$$
Examples of these rings include the three {\bf zonotopal algebras} introduced and studied by Holtz and Ron \cite{Zonotopal}:
\begin{align*}
\text{the {internal zonotopal algebra}}\; \cR_-(\cA) &= R_{\cA,-2},\\
\text{the {central zonotopal algebra}}\; \cR(\cA) &= R_{\cA,-1},\\
\text{and the {external zonotopal algebra}}\; \cR_+(\cA) &= R_{\cA,0}.
\end{align*}

\begin{remark}
The external zonotopal algebra $\cR_+(\cA)$ was originally defined by Postnikov, Shapiro, and Shapiro \cite{PSS}.
Their primary motivation came from the case where $\cA$ arises from the Coxeter arrangement of a Lie algebra via Remark \ref{hyperplanes}, in which case they showed that the
external zonotopal algebra is isomorphic to the subalgebra of differential forms on the flag manifold generated by curvature forms of Hermitian line bundles. 
\end{remark}

\begin{remark}\label{zone}
The word ``zonotopal'' derives from the special case where $\cA$ is {\bf unimodular}:  
this means that $L^*$ is equipped with a lattice $\Lambda$,
the projection $L\to V_e\cong \F$ is given by a primitive element $\chi_e\in\Lambda$, and any subset of the vectors $\{\chi_e\mid e\in E\}$
spans a saturated sublattice of $\Lambda$.
Consider the {\bf zonotope} 
$$Z(\cA) := \left\{\sum c_H \chi_H\bigmid \text{$0\leq c_H\leq 1$ for all $H\in\cA$}\right\}\subset\Lambda\otimes\R.$$
We denote its interior by $\becircled Z(\cA)$, and we also consider a translation $Z(\cA,\lambda) := Z(\cA) + \lambda$ by some 
generic\footnote{Here ``generic'' means that $Z(\cA,\lambda)$ contains no lattice points on its boundary.} element 
$\lambda\in\Lambda\otimes\R$.
Intersecting with $\Lambda$, we obtain finite subsets
\begin{align*}
\cZ_+(\cA) &:= Z(\cA)\cap \Lambda,\\
\cZ(\cA,\lambda) &:= Z(\cA,\lambda)\cap \Lambda,\\
\cZ_-(\cA) &:= \becircled Z(\cA)\cap \Lambda.
\end{align*}
Given a finite subset $S\subset L^*$, its {\bf orbit harmonics ring} $\cR(S)$ is defined to be the coordinate ring of the scheme-theoretic limit
$$\lim_{\epsilon\to 0} (\epsilon\cdot S);$$ 
this may also be characterized as the associated graded ring of the filtration on $\F[S]$ whose $i^\text{th}$ piece is the image of $\Sym^{\leq i}L$
under the restriction of functions from $L^*$ to $S$.
Then we have canonical isomorphisms \cite[Theorems 5.10, 3.9, 4.11]{Zonotopal}\footnote{See \cite[Proposition 1.1]{CDBHP} for an alternate proof in the internal case, which extends to the other two cases.}
$$\cR_-(\cA) \cong \cR(\cZ_-(\cA)),\qquad \cR(\cA) \cong \cR(\cZ(\cA,\la)),\and
\cR_+(\cA) \cong \cR(\cZ_+(\cA)).$$
\end{remark}

Let $C_{\cA,k}\subset \Sym L^*$ be the linear dual of $R_{\cA,k}$, also known as the {\bf Macaulay inverse system} of the defining ideal of $R_{\cA,k}$.  
Concretely, we may think of elements of $\Sym L^*$ as polynomials and elements of $\Sym L$
as differential operators, in which the natural pairing between $\Sym L^*$ and $\Sym L$ is given by
applying the differential operator to the polynomial and then evaluating the resulting polynomial at $0\in L$.  With this interpretation, we may write
$$C_{\cA,k} = \left\{f\in \Sym L^*\bigmid  \text{$\a^{m(\a)+k+1}\cdot f = 0$ for all $0\neq \a\in L$}\right\}.$$
Holtz and Ron use the notation $$\cP_-(\cA) := C_{\cA,-2},\qquad \cP(\cA) := C_{\cA,-1}\and \cP_+(\cA) := C_{\cA,0}.$$

\begin{remark}\label{convention}
We will use the convention that the gradings on all of our vector spaces are induced by the inverse action of $\Gm\subset\Aut(\cA)$.
Thus $\Sym L^*$ is generated in degree 1, $\Sym L$ is generated in degree $-1$, and all nontrivial pairings occur between classes in opposite degrees.
This is consistent with thinking of elements of $\Sym L^*$ as polynomials and elements of $\Sym L$ as differential operators.
\end{remark}

The in the central ($k=-1$) case, we have some beautiful additional structure, namely a canonical ring structure on $\cP(\cA)$,
which we now describe.  
For each $e\in E$, choose an isomorphism $V_e\cong\F$, and let $\chi_e\in L^*$ denote the projection from $L$ to $V_e$.
For any subset $S\subset E$, let 
$$\chi_S := \prod_{e\in S} \chi_e\in \Sym L^*.$$
Consider the ideal\footnote{Holtz and Ron call this ideal $\cJ(\cA)$, but this notation would lead to difficulty for us in Sections \ref{sec:deformations} and \ref{sec:sheaves}.  Similarly, the vector space $\cD(\cA^!)$ in the next paragraph is called $\cD(\cA)$ by Holtz and Ron.} $\cJ(\cA^!)\subset \Sym L^*$ generated by polynomials $\chi_S$ for all subsets $B\subset S$ that are dependent
for the matroid $M_{\cA^!}$.  The quotient ring $\Sym L^*/\cJ(\cA^!)$ is
the reduced Stanley--Reisner algebra $\SRbar(\cA^!)$ of the independence complex of $M_{\cA^!}$.\footnote{The Stanley--Reisner algebra is a quotient
of $\Sym V^*$, and the reduced Stanley--Reisner algebra is the quotient of the Stanley--Reisner algebra by the linear system of parameters $L^\perp\subset V^*$,
which is a quotient of $\Sym L^*$.}  
(See Remark \ref{hypertoric} for an interpretation of this ring in the context of hypertoric geometry.)

Let $\cD(\cA^!)\subset \Sym L$ denote the Macaulay inverse system of $\cJ(\cA^!)$, or equivalently the linear dual of $\SRbar(\cA^!)$.  
When $\F=\R$, $\cD(\cA^!)$ is equal to the {\bf Dahmen--Micchelli space}, 
which is the $\Sym L^*$-submodule of $\Sym L$ spanned by polynomials that coincide on some open subset of $L^*$
with the multivariate spline associated with $\cA$ \cite[Theorem 11.37]{DCP}.  
The Dahmen--Micchelli space is of fundamental importance in approximation theory; it was first studied by de Boor and H\"ollig \cite{dBH}, 
and it is one of the main characters in the book \cite{DCP}.  See \cite[Section 1.2]{Zonotopal} for a more detailed survey of the literature on the space $\cD(\cA^!)$.

The following theorem  was proved by several authors 
\cite{AS88,DR90,J93}, and appears implicitly in \cite{BV99} and the 
references therein. See \cite[Corollary 11.23]{DCP} or \cite[Theorem 
  3.8]{Zonotopal} for a nice exposition.

\begin{theorem}\label{central}
We have $\Sym L^* = \cP(\cA) \oplus \cJ(\cA^!)$.  Equivalently, the compositions
$$\cP(\cA) \hookrightarrow \Sym L^* \twoheadrightarrow \SRbar(\cA^!)
\and
\cD(\cA^!)\hookrightarrow \Sym L \twoheadrightarrow \cR(\cA)$$
are isomorphisms, and the graded vector spaces $\cP(\cA)\cong\SRbar(\cA^!)$ are dual to $\cD(\cA^!)\cong\cR(\cA)$.
\end{theorem}

It is natural to ask whether there is an analogue of Theorem \ref{central} in the internal or external settings.
Indeed, Holtz and Ron define homogeneous ideals
$$\cJ_-(\cA^!), \cJ_+(\cA^!) \subset \Sym L^*$$
and prove that we have \cite[Theorems 4.10 and 5.9]{Zonotopal}
$$\cP_-(\cA) \oplus \cJ_-(\cA^!) = \Sym L^* = \cP_+(\cA) \oplus \cJ_+(\cA^!).$$
The drawback to these results is that the ideals $\cJ_\pm(\cA^!)$ depend on choices.  For example, $\cJ_-(\cA^!)$
depends on a choice of ordering of $E$, and the quotient $\Sym L^*/\cJ_-(\cA^!)$ is the reduced Stanley--Reisner
algebra of the broken circuit complex of $M_{\cA^!}$.  (The broken circuit complex, unlike the independence complex, depends on an ordering.) 
The ideal $\cJ_+(\cA^!)$ depends on both an ordering and a choice of basis for $L^*$ that is suitably generic with respect to $\cA^!$.
In particular, the group $\Aut(\cA) = \Aut(\cA^!)$ acts on the rings $\cR_{\pm}(\cA)$
and on their linear duals $\cP_{\pm}(\cA)$, but not on $\cJ_{\pm}(\cA^!)$.

\begin{remark}\label{internal spline}
Like the central zonotopal algebra, 
the internal zonotopal algebra 
is also interesting from the perspective of approximation theory.
When $\cA$ is unimodular (Remark \ref{zone}), 
Holtz and Ron used the box spline associated with $\cA$ to define a canonical map from $\cP_-(\cA)$ to the 
space of functions on $\cZ_-(\cA)$, and Lenz proved that this map is an isomorphism
\cite{Lenz}.\footnote{This is very different from the isomorphism
obtained from \cite[Result 2.3 and Theorem 5.10]{Zonotopal}, 
which depends on a choice of positive definite inner product
on $\Lambda$.  In particular, that isomorphism takes $1\in\cP_-(\cA)$ to the constant function, whereas
the more interesting and canonical isomorphism in Lenz's theorem takes $1\in\cP_-(\cA)$ to the restriction of the box spline to $\cZ_-(\cA)$.}
\end{remark}

\subsection{The Orlik--Terao algebra}\label{sec:OT-intro}
The first half of this paper is devoted to providing a canonical counterpart to Theorem \ref{central} in the internal setting.
As in the previous section, choose isomorphisms $V_e\cong \F$ and let $\chi_e\in L^*$ denote the projection from $L$ to $V_e$.
The {\bf Orlik--Terao algebra} of $\cA$, defined in \cite{OT-OT} and first systematically studied in \cite{Terao-OT}, is the subalgebra of rational functions on $L$ generated by $\{\chi_e^{-1}\mid e\in E\}$.\footnote{If $\chi_e=0$ for some $e$, then the Orlik--Terao algebra is by definition identically zero.\label{OT-loop}}
Consider the map from $L$ to $\OT(\cA)$ taking $\a\in L$ to $\sum \chi_e(\a) \chi_e^{-1}$.  Note that this map does not change if we scale
one of the linear functionals $\chi_e$, and is therefore completely canonical.  
Equivalently, we may regard $\OT(\cA)$ as a quotient of $\Sym V$, and the map from $L$ is given by the inclusion of $L$ into $V_\cA$.
The image of this map is a linear system of parameters \cite[Proposition 7]{PS}, meaning that $\OT(\cA)$ is a free
module over $\Sym L$ of finite rank.\footnote{This statement can fail in positive characteristic, which is why it is crucial that we assume that $\operatorname{char} \F = 0$.}  The quotient of $\OT(\cA)$ by the ideal generated by the image of $L$ is called the {\bf reduced Orlik--Terao algebra},
and denoted $\OTbar(\cA)$.

\begin{remark}\label{hypertoric}
If $\F=\Q$ and $M_\cA$ is loopless, the Orlik--Terao algebra $\OT(\cA)$ is canonically isomorphic to the equivariant intersection cohomology
of the hypertoric variety associated with $\cA$ \cite{TP08}, and the Stanley--Reisner algebra of the independence complex is canonically isomorphic
to the cohomology ring of an orbifold symplectic resolution of that hypertoric variety \cite{HS}.  Passing to the reduced Orlik--Terao algebra
or the reduced Stanley--Reisner algebra corresponds to passing from equivariant to ordinary cohomology.
\end{remark}

We will apply this construction to the Gale dual linear space $\cA^!$.  
The algebra $\OT(\cA^!)$ is a quotient of $\Sym V^*$,
and $\OTbar(\cA^!)$ is a quotient of $\Sym (V^*/L^\perp) \cong \Sym L^*$.
Let $\cK(\cA^!)\subset \Sym L^*$ denote the kernel of the homomorphism to $\OTbar(\cA^!)$.
Let $\cE(\cA^!)\subset \Sym L$ denote the Macaulay inverse system of $\cK(\cA^!)$, consisting of differential operators in $\Sym L$
that pair trivially with the polynomials in $\cK(\cA^!)$.
The following analogue of Theorem \ref{central} is our first main result.

\begin{theorem}\label{internal}
We have $\Sym L^* = \cP_-(\cA) \oplus \cK(\cA^!)$.  Equivalently, the compositions
$$\cP_-(\cA) \hookrightarrow \Sym L^* \twoheadrightarrow \OTbar(\cA^!)
\and
\cE(\cA^!)\hookrightarrow \Sym L \twoheadrightarrow \cR_-(\cA)$$
are isomorphisms, and the graded vector spaces $\cP_-(\cA)\cong\OTbar(\cA^!)$ are dual to $\cE(\cA^!)\cong\cR_-(\cA)$.
\end{theorem}

\begin{remark}
While the statement of Theorem \ref{internal} is elementary, 
our proof requires the theory of minimal extension sheaves on the poset of flats of the arrangement $\cA^!$, developed
in \cite{TP08}.
\end{remark}

\subsection{An application to graphs}\label{graph-application}
Theorem \ref{internal} has a concrete application outside of the world of zonotopal algebras, which we now describe.
Let $\Gamma$ be a connected graph with vertex set $[n]$.
Let $$\Conf(SU(2),\Gamma) := \{\kappa:[n]\to SU(2)\mid \text{$\kappa(i)\neq\kappa(j)$ if $i$ is adjacent to $j$ in $\Gamma$}\},$$
and let $X(SU(2),\Gamma)$ be the quotient of $\Conf(SU(2),\Gamma)$ by the action of $SU(2)$ by simultaneous left translation.

\begin{remark}
This space can be identified with something more familiar in the case of the complete graph $K_n$.
Since $SU(2)\setminus\{\id\}$ is homeomorphic to $\R^3$,
we may use the translation action of $SU(2)$ to move $\kappa(n)$ to the identity and thus identify $X(SU(2),K_n)$ with the configuration space
of $n-1$ distinct labeled points in $\R^3$.  The advantage of writing it 
as $X(SU(2),K_n)$ is that it reveals the fact that we have an action of 
$\mathfrak{S}_n$ and not just $\mathfrak{S}_{n-1}$. The $\mathfrak{S}_{n-1}$ action on the cohomology 
groups of $X(SU(2),K_n)$ gives (up to a sign) the Eulerian 
representations $\mathfrak{S}_{n-1}$, and the $\mathfrak{S}_n$ action 
gives (up to a sign) Whitehouse's lifts of the Eulerian representations 
to $\mathfrak{S}_n$ \cite[Proposition 2]{ER19}.
For arbitrary $\Gamma$, the space $X(SU(2),\Gamma)$ is more mysterious; it may be regarded as a finite approximation to a product of $g$ copies of the $SU(2)$ affine Grassmanian, where
$g$ is the first Betti number of $\Gamma$ \cite[Remark 1.12]{CDBHP}.
\end{remark}

Let $\cA_\Gamma^! := (E, C_1(\Gamma;\mathbb{Q}), H_1(\Gamma;\mathbb{Q}))$, where $E$ is the set of edges and $C_1(\Gamma;\mathbb{Q})$ is the space of cellular 1-chains.
Let $\cA_\Gamma$ be the Gale dual of $\cA_\Gamma^!$.  If $\Gamma$ is simple, then $\cA_\Gamma$ is the linear space associated 
with the corresponding graphical hyperplane arrangement, that is, the subarrangement
of the Coxeter arrangement for $\mathfrak{sl}_n$ indexed by edges of $\Gamma$.
The following statement was conjectured in \cite[Conjecture 2.16]{MPY}.

\begin{corollary}\label{graph}
There exists a degree-halving isomorphism of graded $\Aut(\Gamma)$-representations $$H^*(X(SU(2),\Gamma);\mathbb{Q})\cong \OTbar(\cA_\Gamma).$$
\end{corollary}

\begin{proof}
By \cite[Theorem 1.9]{CDBHP}, there exists an $\Aut(\cA_\Gamma)$-equivariant vector space isomorphism 
$H^*(X(SU(2),\Gamma);\mathbb{Q})\cong \cR_-(\cA_\Gamma^!)$
taking classes of degree $2i$ to classes of degree $-i$.
Combining this with Theorem \ref{internal}, we see that $H^{2i}(X(SU(2),\Gamma);\mathbb{Q})$ is dual to $\OTbar^i(\cA_\Gamma)$ 
as a representation of $\Aut(\Gamma)\subset\Aut(\cA_\Gamma)$.
The corollary then follows from the fact that finite dimensional
representations of finite groups over $\Q$ are (noncanonically) self-dual.
\end{proof}

\begin{remark}
The case of the complete graph is already very interesting, and was the primary motivation for the conjecture in \cite{MPY}.
For this special case, Corollary \ref{graph} was proved by Pagaria \cite{Pagaria} by studying the generating functions
for the symmetric functions obtained by applying the Frobenius characteristic map to the graded pieces of the representations.
This is a beautiful proof, but it does not provide any canonical 
isomorphism or duality, nor does it extend to other graphs. 
\end{remark}

\begin{remark}
There are many examples of linear spaces over $\C$
for which $\cR_-(\cA)$ and $\OTbar(\cA^!)$ are dual but not isomorphic as graded representations of a finite group of symmetries.
For example, see \cite{Berget-complex-reflection} for an equivariant study of $\cR_-(\cA)$
when $\cA^!$ is the linear space associated with the reflection hyperplanes for the complex reflection group $G(m,1,n)$.
Theorem \ref{internal} allows us to translate this into a study of the reduced Orlik--Terao algebra of $\cA^!$.
\end{remark}

\subsection{The Schubert variety}
Assume now that the ground field $\F$ is algebraically closed.
For all $e\in E$, let $$\overline{V_e} := V_e \sqcup \{\infty\} = \P(V_e \oplus \F) \cong \P^1.$$
The {\bf Schubert variety} $Y_\cA$ is defined as the closure of $L$ inside of
$\PA := \prod \overline {V_e}$.
The name Schubert variety indicates an analogy with Schubert varieties in Lie theory; the most important shared feature is that both varieties admit affine pavings.
The translation action of $L$ on itself extends to an action on $Y_\cA$, which has finitely many orbits 
indexed by flats of $M_\cA$, each of which is isomorphic to an affine space (Lemma \ref{strat}).

Let $F$ be a corank 1 flat.  The corresponding orbit in $Y_\cA$ has codimension 1, and we denote its closure by $D_F$.
Let $\a_F\in L$ span the line in $L$ corresponding to $F$, which is the stabilizer of the orbit, and let
$m(F) := m(\a_F)$.  
For any integer $k$, consider the Weyl divisor $$D_k := \sum_F (m(F) + k) D_F$$
along with the associated sheaf $\cO_{Y_\cA}(D_k)$.\footnote{We are 
using the fact that $Y_{\cA}$ is normal (Theorem~\ref{rational}) to make 
sense of $\cO_{Y_\cA}(D_k)$.}  Concretely, a section of $\cO_{Y_\cA}(D_k)$ is a rational function on $Y_\cA$ that is regular on $L$ and has a pole of order at most $m(F)+k$
on $D_F$.  In particular, we have a canonical inclusion $H^0(\cO_{Y_\cA}(D_k))\subset H^0(\cO_L) \cong \Sym L^*$.  
Our second main result identifies this space of sections with the Macaulay
inverse system $C_{\cA,k}$.  Recall that $R_{\cA,k}$ and $C_{\cA,k}$ are only defined when $k\geq -(\rho+1)$, where $\rho$ is the smallest possible value of $m(\a)$.

\begin{theorem}\label{Schubert}
For all $-(\rho + 1) \leq k \leq 0$, we have $H^0(\cO_{Y_\cA}(D_k)) = C_{\cA,k}$ 
as subspaces of $\Sym L^*$.\footnote{
When $k$ is positive, $H^0(\cO_{Y_\cA}(D_k))$ is equal to the slightly larger space $C_{\!\cA,k}'$ appearing in \cite{ArPo-correction}.}
\end{theorem}

\subsection{Syzygies}\label{sec:syzygies-intro}
Our primary motivation for Theorem \ref{Schubert} is that these cohomology groups appear naturally in the minimal resolution of the structure sheaf of $Y_\cA$ on the ambient variety $\PA$.
For any $S\subset E$, let $V_S := \bigoplus_{e\in S} V_e$, let $\pi_S:V\to V_S$ denote the projection, and let $L_S := \pi_S(L)\subset V_S$.  We define the {\bf localization}
$\cA_S:= (S, V_S, L_S)$ and the {\bf contraction} $\cA^S := (E\setminus S, V_{E\setminus S}, L \cap V_{E\setminus S})$, so that $(\cA^!)_S = (\cA^{E\setminus S})^!$.  At the level of matroids, these operations correspond to restriction to $S$
and contraction of $S$, respectively.  Let $\pi_e:\PA\to\overline{V}_e$
denote the coordinate projection, and let $$\cO_{\PA}(-1_S) := \bigotimes_{e\in S}\pi_e^*\cO_{\overline{V}_e}(-1).$$
Let $M = M_{\cA^!}$ be the matroid associated with the Gale dual of $\cA$.  
For any vector space $V$, let $\det(V) := \wedge^{\dim V}V$.

\begin{theorem}\label{resolution}
There exists an $\Aut(\cA)$-equivariant minimal free resolution $$\cdots\to \cQ_2\to\cQ_1\to\cQ_0\to\cO_{Y_\cA}\to 0$$
of the structure sheaf $\cO_{Y_\cA}$ on $\P_\cA$, and we have
\begin{equation*}\label{syzygies}\cQ_i \;\;\;\;\;\cong \bigoplus_{\operatorname{crk}_M(E\setminus S) = i}
\det(L_S^\perp)\otimes H^0\big(\cO_{Y_{\cA_S}}(D_{-2})\big)^*\otimes \cO_{\PA}(-1_S).\footnote{The
action of $\Aut(\cA)$ on the right-hand side of this expression will be made precise in Remark \ref{linearization}.}\end{equation*}
\end{theorem}

\begin{remark}
By Theorem \ref{Schubert}, the factor $H^0\big(\cO_{Y_{\cA_S}}(D_{-2})\big)^*$ in Theorem \ref{resolution} may be replaced by the internal zonotopal
algebra $\cR_-(\cA_S)$.  By Theorem \ref{internal}, it may be replaced by $\OTbar((\cA_S)^!)^*$.
\end{remark}

\begin{remark}\label{flats}
Theorem \ref{resolution} is a categorification of \cite[Theorem 1.5]{ArBoo}, in which the multigraded Betti numbers of $\cO_{Y_\cA}$ are computed.
The novelty of Theorem \ref{resolution} is that it encodes the action of the group $\Aut(\cA)$ on the multiplicity spaces, including the multiplicative
group $\Gm\subset\Aut(\cA)$, which is responsible for the gradings.

Though we have written the resolution in Theorem \ref{resolution} as a direct sum over all subsets $S\subset E$, many of the terms
vanish.  Indeed, $\cR_-(\cA_S)$ is nonzero if and only if $M_{\cA_S}$ has no coloops, which is the case if and only
if $E\setminus S$ is a flat of $M$. 
\end{remark}

We conclude by considering what happens when we tensor the exact sequence in Theorem \ref{resolution} with $\cO(1_E) := \cO(-1_E)^*$ and take cohomology.
All of the sheaves appearing in this new exact sequence have vanishing higher cohomology and $\cO_{Y_\cA}\otimes\cO(1_E) \cong \cO_{Y_\cA}(D_0)$, which leads to the following result
relating internal and external zonotopal algebras.

\begin{corollary}\label{cohomology}
There exists an $\Aut(\cA)$-equivariant exact sequence of vector spaces
$$\cdots\to H^0(\cQ_2(1_E))\to H^0(\cQ_1(1_E))\to H^0(\cQ_0(1_E))\to \cR_+(\cA)^*\to 0,$$
where 
$$H^0(\cQ_i(1_E)) \;\;\;\;\;\; 
\cong \bigoplus_{\operatorname{crk}_{M}(E\setminus S) = i}
\det(L_S^\perp)\otimes \cR_-(\cA_S) \otimes \bigotimes_{e\in E\setminus S}(V_e^* \oplus\F).$$
\end{corollary}

\begin{remark}\label{spanning}
We have noted in Remark \ref{flats} that $\cR_-(\cA_S)$ vanishes unless $E\setminus S$ is a flat of $M$.
If it is a flat, then the total dimension of $\cR_-(\cA_S)$ is equal to $(-1)^{\operatorname{crk}(E\setminus S)} \mu(E\setminus S,E)$, where $\mu$ is the M\"obius function
on the lattice of flats of $M$ \cite[Theorem 5.9(1)]{Zonotopal}.  The total dimension of $\cR_+(\cA)$ is equal to the number of spanning sets of $M$ \cite[Theorem 4.10(1)]{Zonotopal}.  
Thus the exact sequence
in Corollary \ref{cohomology} categorifies the equation
$$ \text{number of spanning sets of $M$}\;\;\; = \sum_{\text{flats $F$}} \mu(F,E) \,2^{|F|}.$$
This identity is obtained via M\"obius inversion from the more familiar identity
$$2^{|E|}\;\;\; = \sum_{\text{flats $F$}} |\{T\subset E\mid \overline{T} = F\}|.$$
\end{remark}

\vspace{\baselineskip}
\noindent
{\em Acknowledgments:}
We are grateful to Andy Berget and Matt Larson for valuable 
conversations. We would also like to thank the creators of Macaulay2, as 
well as Federico Galleto for writing the BettiCharacters package, which 
was helpful in proving Theorem~\ref{resolution}.
The first author thanks the Simons 
Foundation for support.  The second author thanks All Souls College for 
its hospitality during the preparation of this manuscript.
Both authors thank the National Science Foundation.

\section{Modules over \boldmath{$\Sym L^\perp$}}\label{sec:deformations}
The algebra $\OTbar(\cA^!)$ comes with a flat deformation $\OT(\cA^!)$ over the base $(L^\perp)^*\cong V_\cA/L$.
Our strategy for proving Theorem \ref{internal} will be to define a corresponding flat deformation $\cQ_-(\cA)$ of $\cP_-(\cA)$ over the same base,
and to prove that $\cQ_-(\cA)$ is isomorphic to $\OT(\cA^!)$.

\subsection{OT and \boldmath{$\SR_<$}}\label{sec:OTSR}
Fix for each $e\in E$ a nonzero vector $v_e\in V_e$, and let $w_e\in V_e^*$ be the unique vector that pairs to 1 with $v_e$.  
Let $\chi_e\in L^*$ be the linear functional obtained by projecting onto $V_e$ and pairing with $w_e$,
and let $\psi_e\in (L^\perp)^*$ be the linear functional obtained by projecting onto $V_e^*$ and pairing with $v_e$.
We have surjections $$\Sym V\to \OT(\cA)\and \Sym V^*\to \OT(\cA^!)$$ taking $v_e$ to $\chi_e^{-1}$ and $w_e$ to $\psi_e^{-1}$.
We will only be interested in $\OT(\cA^!)$, but we provide the two constructions in parallel to make them more transparent.

\begin{remark}
Note that scaling $v_e$ has the effect of scaling $\psi_e$ and inverse scaling $w_e$ and $\chi_e$, so these surjections are in fact independent of any choices. 
This observation makes it clear why $\OT(\cA)$ is canonically a quotient of $\Sym V$ rather than $\Sym V^*$.
\end{remark}

Let $\cG(\cA^!)\subset \Sym V^*$ be the kernel of the homomorphism to $\OT(\cA^!)$.
For any element $\a \in L$, let
$\Supp(\a) := \{e\mid \chi_e(\a) \neq 0\}$, so that $m(\a) = |\Supp(\a)|$.
Since $\a$ vanishes on $L^\perp$, so does the rational function
$$\a\prod_{e\in \Supp(\a)} \psi_e^{-1}\;\;\; = \sum_{e\in \Supp(\a)} \chi_e(\a) \prod_{f\in\Supp(\a)\setminus e} \psi_f^{-1}.$$
Thus we have a corresponding element $$h_\a \;\;\;:= \sum_{e\in \Supp(\a)} \chi_e(\a) \prod_{f\in\Supp(\a)\setminus e} w_f\;\;\;\;\in\;\;\;\;\cG(\cA^!).$$
If there exists some $\a\in L$ with $\Supp(\a) = \{e\}$, then $h_\a$ is a nonzero constant, and therefore $\cG(\cA^!) = \Sym V^*$. 
This happens if and only if $\psi_e=0$, and therefore agrees with the 
convention in Footnote \ref{OT-loop} that $\OT(\cA^!)=0$ in such cases.

Fix a total ordering $<$ of $E$, along with a monomial order on $\Sym V^*$ with the property that
$w_e < w_f$ if and only if $e < f$.  For any $0\neq g\in \Sym V^*$, we define the initial term $\operatorname{in}_<(g)\in \Sym V^*$
to be the smallest monomial appearing in $g$ with respect to our monomial order.
Let $$\cG_<(\cA^!) := 
\left\langle  \operatorname{in}_<(h_\a) \bigmid \a\in L\right\rangle\subset \Sym V^*.$$
This ideal is called the Stanley--Reisner ideal of the broken circuit complex of $M_{\cA^!}$, and the quotient
$$\SR_<(\cA^!) := \Sym V^* / \cG_<(\cA^!)$$ is called the Stanley--Reisner algebra of the broken circuit complex. 

We call $0\neq \a\in L$ {\bf minimal} if its support does not strictly contain the support of another nonzero element of $L$, or equivalently if the support of $\a$
is a circuit of $M_{\cA^!}$.  It is clear that $\cG_<(\cA^!)$ is generated by $\operatorname{in}_<(h_\a)$ for minimal $\a$.
Since $h_\a\in\cG(\cA^!)$, we {\em a priori} have an inclusion $\cG_<(\cA^!) \subset \operatorname{in}(\cG(\cA^!))$.
In fact, this inclusion turns out to be an isomorphism; the following result is proved using the theory of geometric vertex decompositions \cite{KMY}.

\begin{theorem}\label{OT-degeneration}{\em \cite[Theorem 4]{PS}}
The set $\{h_\a\mid \text{{\em $\a$ minimal}}\}$ is a universal Gr\"obner basis for $\cG(\cA^!)$.  That is, for any choice of ordering $<$
and any compatible choice of monomial order,
we have $\operatorname{in}(\cG(\cA^!)) = \cG_<(\cA^!)$.
\end{theorem}

The inclusion of $L^\perp$ into $V^*$ makes both $\OT(\cA^!)$ and $\SR_<(\cA^!)$ graded modules over $\Sym L^\perp$, and they are both free
modules by \cite[Propositions 1 and 7]{PS}.
We define the {\bf reduced Orlik--Terao algebra} and {\bf reduced Stanley--Reisner algebra}
$$\OTbar(\cA^!) := \OT(\cA^!) \otimes_{\Sym L^\perp} \F \and \SRbar_<(\cA^!) =: \SR_<(\cA^!) \otimes_{\Sym \cA^!} \F.$$
We may equivalently regard $\OTbar(\cA^!)$ and $\SRbar_<(\cA^!)$ as quotients of $\Sym L^* \cong \Sym (V_\cA^*/L^\perp)$ by the images of $\cG(\cA^!)$ and $\cG_<(\cA^!)$ in $\Sym L^*$.
These images are precisely the ideals $\cK(\cA^!)$ and $\cJ_-(\cA^!)$ from Sections \ref{sec:OT-intro} and \ref{sec:background}.

\subsection{\boldmath{$\cQ_-$}}\label{P-intro-section}
Recall that we have defined
$$\cP_-(\cA) := \left\{f\in \Sym L^*\bigmid  \text{$\a^{m(\a)-1}\cdot f = 0$ for all $0\neq \a\in L$}\right\}.$$
Ardila and Postnikov prove that it is enough to consider minimal $\a$ \cite[Lemma 1]{ArPo-correction}, which may be regarded
as an analogue of the statement that $\cG(\cA^!)$ is generated by the classes $h_\a$ for minimal $\a$.
Our deformation $\cQ_-(\cA)$ of $\cP_-(\cA)$ is defined in the most naive possible way:
$$\cQ_-(\cA) := \left\{f\in \Sym V^*\bigmid \text{$\a^{m(\a)-1}\cdot f = 0$ for all $0\neq \a\in L$}\right\}\subset \Sym V^*.$$
The inclusion of $L^\perp$ into $V^*$ makes $\cQ_-(\cA)$ a graded module over $\Sym L^\perp$. 

\begin{lemma}\label{free}
The space $\cQ_-(\cA)$ is a free module over $\Sym L^\perp$, and we have a canonical isomorphism of graded vector spaces
$$\cP_-(\cA)\cong \cQ_-(\cA) \otimes_{\Sym L^\perp} \F.$$
\end{lemma}

\begin{proof}
Any splitting of the inclusion of $L^\perp$ into $V^*$ induces a graded $\Sym L^\perp$-module isomorphism $\cQ_-(\cA) \cong \cP_-(\cA)\otimes \Sym L^\perp$.
While this isomorphism depends on the choice of splitting, the induced isomorphism
$\cP_-(\cA)\cong \cQ_-(\cA) \otimes_{\Sym L^\perp} \F$ does not.
\end{proof}

\section{Sheaves on the poset of flats}\label{sec:sheaves}
Let $\cL$ be the poset of flats of the matroid $M_{\cA^!}$.  
We will assume throughout this section that $M_{\cA^!}$ is loopless, so the minimal element of $\cL$ is $\emptyset$ and the maximal element is $E$.
Consider the topology on the set $\cL$ in which $U\subset\cL$ is open if and only if it is downward-closed: that is, $F\in U$
whenever $F\leq G\in U$.  The minimal open neighborhood of $G$ is the set $U_G := \{F\mid F\leq G\}$.
We can also think of $\cL$ as the set of objects of a thin category in which $|\Hom(G,F)| = 1$ if $F\leq G$ and 0 otherwise.
Given a sheaf $\cF$ on the topological space $\cL$, we can produce a functor on the category $\cL$
by sending a flat $G$ to the stalk $\cF_G = \cF(U_G)$.  This process is reversible, so sheaves are in canonical bijection
with functors \cite[Proposition 1.1]{TP08}. 

For any flat $F$, recall the that we have defined the localization $(\cA^!)_F = (F, V^*_F, (L^\perp)_F)$ in Section \ref{sec:syzygies-intro}.
For ease of notation, we will write $\cA^!_F$ for $(\cA^!)_F$ and $L^\perp_F$ for $(L^\perp)_F = (L \cap V_F)^\perp$, but we stress that the Gale duality comes first and the localization second.
We define a sheaf $\cO$ on by putting
$$\cO_F := \Sym L^\perp_F$$
and taking $\cO_G\to\cO_F$ to be the canonical projection for any $F\leq G$.  In particular, we have $\cO_E = \Sym L^\perp$ and $\cO_\emptyset = \F$.
We call $\cO$ the {\bf structure sheaf} of $\cL$, and we will be interested in sheaves of graded $\cO$-modules.

\subsection{Minimal extension sheaves}
Our main technical tool in this paper will be the theory of minimal extension sheaves on $\cL$, developed in \cite{TP08}.

\begin{proposition}\label{mes}{\em \cite[Proposition 1.10 and Corollary 3.7]{TP08}}
There exists a sheaf $\cF$ of graded $\cO$-modules on $\cL$ with the following properties:
\begin{itemize}
\item For every $F\in\cL$, $\cF_F$ is a free module over $\cO_F$.
\item The sheaf $\cF$ is flabby:  for every open $U\subset \cL$, the restriction $\cF(\cL)\to \cF(U)$ is surjective.
\item We have $\cF_{\emptyset} \cong \cO_{\emptyset} \cong \F$.
\end{itemize}
The sheaf $\cF$ is unique up to isomorphism.  Furthermore, the only endomorphisms of $\cF$ are those given by scalar multiplication.
\end{proposition}

\begin{remark}
In \cite{TP08}, it is only asserted that all automorphisms of $\cF$ are given by scalar multiplication, rather than all endomorphisms.
However, the proof of the stronger statement is identical; see \cite[Lemma 1.16]{TP08} and \cite[Remark 1.18]{BBFK}.
\end{remark}

Any sheaf $\cF$ satisfying the conditions of Proposition \ref{mes} is called a {\bf minimal extension sheaf}.

\begin{corollary}\label{canonical}
Suppose that $\cF$ and $\cF'$ are minimal extension sheaves.
Any sheaf homomorphism $\varphi:\cF\to\cF'$
with the property that $\varphi_\emptyset:\cF_\emptyset\to \cF'_\emptyset$ is nonzero is necessarily an isomorphism.
\end{corollary}

\begin{proof}
By Proposition \ref{mes}, $\cF$ and $\cF'$ are isomorphic, and there is a unique
isomorphism $\psi:\cF'\to\cF$ with the property that $\psi_\emptyset = \varphi_\emptyset^{-1}$.
Then $\varphi\circ\psi$ is an endomorphism of $\cF$ and $\psi\circ\varphi$ is an endomorphism of $\cF'$, both of which restrict to the identity
at $L^\perp$.  By the last sentence of Proposition \ref{mes}, they are both equal to the identity, so $\varphi = \psi^{-1}$.
\end{proof}

\begin{remark}\label{it ain't easy bein' flabby}
For a sheaf $\cF$ to be flabby, it is not enough for the stalk maps $$\cF_E = \cF(\cL)\to \cF(U_F) = \cF_F$$ to be surjective for all $F\in \cL$.
Rather, the map $\cF(\cL)\to\cF(U)$ has to be surjective for {\em every} open set $U$, including those that are not of the form $U_F$.
For example, the structure sheaf $\cO$ is almost never flabby \cite[Example 1.14]{TP08}, and is therefore almost never a minimal extension sheaf.
\end{remark}

\subsection{OT and \boldmath{$\SR_<$}}
For any flat $F$ of the matroid $M_{\cA^!}$, we will be interested in the rings $\OT(\cA^!_F)$ and $\SR_<(\cA^!_F)$, both of which are algebras over $\cO_F = \Sym L^\perp_F$.
Recall that the localization $\cA^!_F$ is Gale dual to the contraction $\cA^{E\setminus F} = (F, V_F, L\cap V_F)$, therefore
we have $$\OT(\cA^!_F) \cong \Sym V^*_F / \cG(\cA^!_F)\and \SR_<(\cA^!_F) \cong \Sym V^*_F / \cG_<(\cA^!_F),$$
where the generators of $\cG(\cA^!_F)$ are indexed by minimal elements $\a\in L \cap V_F$.
These are the same as minimal elements $\a\in L$ with $\Supp(\a)\subset F$.

Given a pair of flats $F\leq G$, we have a homomorphism $$\rho_{FG}:\Sym V_G^*\to \Sym V_F^*$$ induced by the projection of $V^*_G$ onto $V^*_F$.
By definition, $\rho_{FG}$ restricts to a homomorphism from $\Sym L^\perp_G\subset \Sym V^*_G$ to $\Sym L^\perp_F\subset \Sym V^*_F$.

\begin{lemma}\label{OT-sheaf}
We have $$\rho_{FG}(\cG(\cA^!_G))\subset \cG(\cA^!_F)\and \rho_{FG}(\cG_<(\cA^!_G))\subset \cG_<(\cA^!_F),$$
thus $\rho_{FG}$ descends to graded algebra homomorphisms
$$\OT(\cA^!_G)\to \OT(\cA^!_F) \and \SR_<(\cA^!_G)\to \SR_<(\cA^!_F).$$
\end{lemma}

\begin{proof}
Suppose that $\a\in L$ is minimal and $\Supp(\a)\subset G$.  If $\Supp(\a)\subset F$, then 
$$\rho_{FG}(h_C) = h_\a\and \rho_{FG}(\operatorname{in}_<(h_\a)) = \operatorname{in}_<(h_\a).$$
Otherwise, the fact that $\Supp(\a)$ is a circuit of $M_{\cA^!}$ and $F$ is a flat of $M_{\cA^!}$ not containing $\Supp(\a)$ implies that $F$ is missing at least two elements of $\Supp(\a)$, 
and therefore we can conclude that $\rho_{F,E}(h_\a)=0=\rho_{FG}(\operatorname{in}_<(h_\a))$.
\end{proof}

By Lemma \ref{OT-sheaf}, we may define sheaves of $\cO$-modules $\OT$ and $\SR_<$ on $\cL$ whose stalks at $F$ are $\OT(\cA^!_F)$ and $\SR_<(\cA^!_F)$, 
with restriction maps given by the homomorphisms $\rho_{FG}$.

\begin{proposition}{\em \cite[Propositions 3.9 and 3.11]{TP08}}\label{OT-min}
The sheaves $\OT$ and $\SR_<$ are both minimal extension sheaves.
\end{proposition}

\begin{remark}
Propositions \ref{mes} and \ref{OT-min} have the surprising consequence that there is a unique
isomorphism of $\cO$-modules $\OT\cong\SR_<$ that restricts to the identity at the flat $\emptyset$, and therefore
a canonical isomorphism of graded $\Sym L^\perp$-modules $$\OT(\cA^!) = \OT_E\cong(\SR_<)_E = \SR_<(\cA^!).$$  This is not a ring isomorphism,
and it is not easy to give an explicit formula for it.
\end{remark}

\subsection{\boldmath{$\cQ_-$}}
In this section, we will be interested in the graded $\Sym L^\perp_F$-module 
$$\cQ_-(\cA^{E\setminus F}) = \left\{f\in \Sym V_F^*\bigmid \text{$\a^{m(\a)-1}\cdot f = 0$ for all $0\neq \a\in L\cap V_F$}\right\}\subset \Sym V_F^*.$$

\begin{lemma}\label{IZ-sheaf}
If $F\leq G$ are flats, then $\rho_{FG}:\Sym V^*_G\to \Sym V^*_F$ takes $\cQ_-(\cA^{E\setminus G})$ to $\cQ_-(\cA^{E\setminus F})$.
\end{lemma}

\begin{proof}
If $f\in \cQ_-(\cA^{E\setminus G})$ and $0\neq \a\in L\cap V_F$, then
$$\a^{m(\a)-1}\cdot \rho_{FG}(f) = \rho_{FG}(\a^{m(\a)-1}\cdot f) = \rho_{FG}(0) = 0.$$
This completes the proof.
\end{proof}

By Lemma \ref{IZ-sheaf}, we may define a sheaf $\cQ_-$ on $\cL$ whose stalk at $F$ is equal to $\cQ_-(\cA^{E\setminus F})$, with restriction maps given by the homomorphisms $\rho_{FG}$.
We would like to show that $\cQ_-$ is a minimal extension sheaf.  The fact that the stalks are free modules is proved in Lemma \ref{free} and the fact that the stalk at the empty flat
is isomorphic to $\F$ is trivial.  However, the flabbiness property is not obvious.  While it is easy to show that the map $\rho_{FE}:\cQ_-(\cA)\to\cQ_-(\cA^{E\setminus F})$ is surjective, this is not enough (see Remark \ref{it ain't easy bein' flabby}).  This issue will be resolved in the next section by
making use of a result from \cite{Zonotopal}.

\subsection{Isomorphisms}
The inclusion of $\cQ_-(\cA^{E\setminus F})$ into $\Sym V^*_F$, followed by the projection of $\Sym V^*_F$ onto $\OT(\cA^!_F)$ or $\SR_<(\cA^!_F)$, define graded $\Sym L^\perp_F$-module homomorphisms
$$\varphi_F:\cQ_-(\cA^{E\setminus F})\to \OT(\cA^!_F)\and (\varphi_<)_F:\cQ_-(\cA^{E\setminus F})\to \SR_<(\cA^!_F).$$
These morphisms are all compatible with the maps $\rho_{FG}$, and therefore they define sheaf homomorphisms
$$\varphi:\cQ_-\to\OT\and \varphi_<:\cQ_-\to\SR_<.$$

\begin{proposition}\label{SR-isom}
The homomorphism $\varphi_<:\cQ_-\to \SR_<$ is an isomorphism.
\end{proposition}

\begin{proof}
Fix a flat $F$, and consider the graded $\Sym L^\perp_F$-module homomorphism $$(\varphi_<)_F:\cQ_-(\cA^{E\setminus F})\to \SR_<(\cA^!_F).$$
Since the source and the target are both free modules, it is sufficient to prove that
it becomes an isomorphism after tensoring over $\Sym L^\perp_F$ with $\F$.  This yields the map
$$\cP_-(\cA^{E\setminus F})\to \SRbar_<(\cA^!_F),$$
which is an isomorphism by \cite[Theorem 5.9(6)]{Zonotopal}.
\end{proof}

\begin{corollary}\label{MES}
The sheaf $\cQ_-$ is a minimal extension sheaf, and $\varphi:\cQ_-\to \OT$ is an isomorphism.
\end{corollary}

\begin{proof}
The first statement follows from Propositions \ref{OT-min} and \ref{SR-isom}, and the second follows from 
the first statement, Proposition \ref{OT-min}, and Corollary \ref{canonical}.
\end{proof}

\begin{proof}[Proof of Theorem \ref{internal}]
If $M_{\cA^!}$ has a loop $e$, then there exists $\a\in L$ with $\Supp(\a) = \{e\}$.  It follows that $\cP_-(\cA) = 0$ and $\cK(\cA^!) = \Sym L^*$, and therefore the theorem holds.
Thus we may assume that $M_{\cA^!~}$ is loopless, which allows us to apply the machinery of Section \ref{sec:sheaves}.

By Corollary \ref{MES}, the map $\varphi_E:\cQ_-(\cA)\to\OT(\cA^!)$ is an isomorphism.  Tensoring over $\Sym L^\perp$ with $\F$, we see the the composition
$\cP_-(\cA) \hookrightarrow \Sym L^* \twoheadrightarrow \OTbar(\cA^!)$ is an isomorphism.
\end{proof}

\subsection{An example}
We conclude this section with an example of Theorem~\ref{internal}.
We have intentionally chosen an example that has proved to be deceptive in the past.
The space $\cP_-(\cA)$ includes all of the monomials $\chi_S = \prod_{e\in S} \chi(e)\in \Sym L^*$ for those subsets $S$ such that
every circuit of $M_{\cA^!}$ contains at least two elements of $E\setminus S$, and Holtz and Ron conjectured that $\cP_-(\cA)$ is spanned
by such monomials \cite[Conjecture 6.1]{Zonotopal}.  A proof of this conjecture was published in \cite{ArPo} and later retracted in \cite{ArPo-correction}, along with a counterexample.
The example that we consider here is precisely the counterexample appearing in \cite{ArPo-correction}.  One can regard this as the first example for which the space
$\cP_-(\cA)$ is trickier than one might expect.

Let $\Gamma$ be the graph obtained from the complete graph on three vertices by doubling each edge, and
let $$\cA := \cA_\Gamma^! = (E,C_1(\Gamma; \Q),H_1(\Gamma; \Q))$$
be the corresponding cographical linear space (see Section~\ref{graph-application}).
We will label the edges with the labels $e_1$, $e_2$, $f_1$, $f_2$, $g_1$, and $g_2$, so that $\{e_1,e_2\}$, $\{f_1,f_2\}$ and $\{g_1,g_2\}$ are the three parallel pairs.
We will orient them so that each parallel pair is oppositely oriented and $\{e_1, f_1, g_1\}$ forms an oriented cycle.  Then $L = H_1(\Gamma; \Q)$ is spanned by the classes
$$\ell_e:= v_{e_1} + v_{e_2},\qquad \ell_f:= v_{f_1} + v_{f_2},\qquad \ell_g:= v_{g_1} + v_{g_2},\and \ell := v_{e_1}+v_{f_1}+v_{g_1}.$$ 
The defining ideal
$$\left\langle \a^{m(\a)-1}\bigmid 0\neq \a\in H_1(\Gamma;\Q)\right\rangle$$
of the internal zonotopal algebra turns out to be generated by the elements 
$\ell_e$, $\ell_f$, $\ell_g$, and $\ell^2$,
thus $\cR_-(\cA) \cong \Q[\ell]/\ell^2$.
We therefore have $\cP_-(\cA) = \Q\{1,c\}$, where $c\in L^*= H^1(\Gamma;\Q)$ is the class that pairs trivially with $\ell_e$, $\ell_f$, and $\ell_g$, and pairs to 1 with $\ell$.
We have a surjection from the space $V^* = C^1(\Gamma;\Q) = \Q\{w_{e_1}, w_{e_2}, w_{f_1}, w_{f_2}, w_{g_1}, w_{g_2}\}$ to $H^1(\Gamma;\Q) = L^*$, but $c$ is not the image of any basis element.
Rather, $c$ is the image of the element $(w_{e_1} - w_{e_2} + w_{f_1} - w_{f_2} + w_{g_1} - w_{g_2})/3$.

On the dual side, we have the Orlik--Terao algebra
$$\OT(\cA^!) \cong \Q[w_{e_1}, w_{e_2}, w_{f_1}, w_{f_2}, w_{g_1}, w_{g_2}]/\langle w_{e_1}+w_{e_2}, w_{f_1}+w_{f_2}, w_{g_1}+w_{g_2}, w_{e_1}w_{f_1} + w_{e_1}w_{g_1} + w_{f_1}w_{g_1}\rangle.$$
The linear space $L^\perp\subset V^*$ is generated by the classes $w_{e_1} - w_{e_2} -w_{f_1} + w_{f_2}$ and $w_{f_1} - w_{f_2} - w_{g_1} + w_{g_2}$, 
thus we have the reduced Orlik--Terao algebra
$$\OTbar(\cA^!) \cong \OT(\cA)/\langle w_{e_1} - w_{e_2} -w_{f_1} + w_{f_2}, w_{f_1} - w_{f_2} - w_{g_1} + w_{g_2}\rangle \cong \Q[w]/\langle w^2\rangle,$$
where $w := w_{e_1} = w_{f_1} = w_{g_1} = -w_{e_2} = -w_{f_2} = -w_{g_2}$.
The isomorphism $\cP_-(\cA) \cong \OTbar(\cA^!)$ of 
Theorem~\ref{internal} coincides with the map
\[
\cP_-(\cA) = \Q\{ 1, c\} \to \Q[w]/\langle w^2\rangle \cong \OTbar(\cA^!)
\]
that sends $c$ to $2w$.

\section{Schubert varieties}
The goal of this section is to prove Theorems \ref{Schubert} and \ref{resolution}.

\subsection{The local geometry of \boldmath{$Y_\cA$}}
Recall that, by definition, $Y_\cA$ is a subvariety of $\P_\cA = 
\prod_{e\in E} \overline{V}_e$. Because $Y_\cA$ is the closure of a 
linear space, the $\N^E$-graded multidegree of $Y_\cA$ is multiplicity 
free \cite{ArBoo,Li}. The following statement is a consequence of the 
main theorem in \cite[Theorem 1]{Brion} on multiplicity free subvarieties of flag 
varieties; see the proof of \cite[Theorem 4.3]{BeFi} for a discussion of rational singularities.

\begin{theorem}\label{rational}
  The Schubert variety $Y_{\cA}$ has rational singularities, and is therefore 
  normal and Cohen--Macaulay.
\end{theorem}

The additive group $V$ acts on $\PA$, and the additive subgroup $L\subset V$ acts on $Y_\cA$.
For any subset $F\subset E$, let $p_F \in \PA$ be the point defined by putting
$$(p_F)_e := \begin{cases}
0 &\text{if $e\in F$}\\
\infty &\text{if $e\notin F$.}
\end{cases}$$

\begin{lemma}\label{strat}{\em \cite[Proposition 4.10]{KLS}}
The following statements hold.
\begin{itemize}
\item We have $p_F\in Y_\cA$ if and only if $F$ is a flat of $M_\cA$.\footnote{Warning:  All of the flats in Section \ref{sec:sheaves} were 
flats of $M_{\cA^!}$, whereas now we are considering flats of $M_{\cA}$.}
\item The stabilizer in $L$ of $p_F$ is $L\cap V_{E\setminus F}$.
\item Every element of $Y_\cA$ lies in the orbit of a unique point $p_F$.  That is, we have
$$Y_\cA \;= \bigsqcup_{\text{$F$ a flat}} L\cdot p_F \;\;\cong\; \bigsqcup_{\text{$F$ a flat}} L/(L\cap V_F)\;\;\cong\; \bigsqcup_{\text{$F$ a flat}} L_F.$$
\item Given two flats $F$ and $G$, $L_F$ is contained in the closure of $L_G$ if and only if $F\leq G$.
\end{itemize}
\end{lemma}

Lemma \ref{strat} describes a stratification of $Y_\cA$.  Our next goal is to describe $Y_\cA$ in the neighborhood of a stratum.
For any flat $F$, we have a closed embedding $\iota_F:Y_{\cA^F}\to Y_\cA$ given by the formula
$$\iota_F(p)_e := \begin{cases}
0 &\text{if $e\in F$}\\
p_e &\text{if $e\notin F$.}
\end{cases}$$
Let $$U_F :=  \bigsqcup_{F\leq G} L_G = Y_\cA \setminus \bigcup_{F\not\leq G} \overline{L_G}.$$ 
The second description shows that $U_F$ is an open subscheme of $Y_\cA$.

Choose an arbitrary section $s:L_F\to L$ of the projection of $L$ onto $L_F$.  For any flat $G\geq F$, 
composing $s$ with the projection of $L$ onto $L_G$ gives a splitting $s_G:L_F\to L_G$ of the projection of $L_G$ onto $L_F$.

\begin{lemma}\label{neighborhood}
We have an isomorphism $\varphi:L_F\times Y_{\cA^F}\to U_F$ given by the formula
$$\varphi(\a,p) := s(\a)\cdot \iota_F(p).$$
\end{lemma}

\begin{proof}
The poset of flats of $M_{\cA^F}$ is isomorphic to the interval $[F,E]$ in the poset of flats of $M_{\cA}$, so both the source and target
have stratifications indexed by the same posets.  In the source, the stratum corresponding to a flat $G\geq F$ is isomorphic to
$L_F \times L_G/L_F$, and in the target it is isomorphic to $L_G$.  
The map $\varphi$ is compatible with these stratifications, and the restriction of $\varphi$ to the stratum $L_F \times L_G/L_F$ coincides with the isomorphism induced by the splitting $s_G$.
Since $\varphi$ restricts to a stratum-by-stratum isomorphism, it is a bijection.
The variety $Y_{\cA}$ is normal by Theorem \ref{rational}, therefore so is $U_F$, thus Zariski's main theorem implies that $\varphi$ is an isomorphism.
\end{proof}

\begin{proof}[Proof of Theorem \ref{Schubert}]
Consider the special case of Lemma \ref{neighborhood} where $F$ is a flat of corank 1.  In this case, the subspace $L^F := L\cap V_{E\setminus F}\subset L$ is 1-dimensional,
we have $Y_{\cA^F} \cong \overline{L^F} \cong \P^1$, and for any $\a\in L_F$, $\varphi(\a,L^F) = s(\a) + L^F$ is an affine translate of $L^F$.

Let $f$ be a rational function on $Y_\cA$ whose restriction to $L$ is regular.  Then 
$$\operatorname{ord}_{D_F} f = \operatorname{ord}_{L_F} (f|_{U_F}) = \operatorname{ord}_{L_F \times\{\infty\}} \varphi^*(f|_{U_F}),$$
which is equal to the degree of $f$ when restricted to a generic affine translate of the line $L^F \subset L$.
It follows that  $H^0(\cO_{Y_\cA}(D_k))$ is equal to the set of polynomials $f\in \Sym L^*$ such that, for all corank 1 flats $F$, the restriction of $f$ to 
any affine translate of $L^F$ has degree at most $m(F)+k$.
This is equal to $C_{\cA,k}$ by \cite[Proposition 2.6]{ArPo} and \cite[Lemma 1]{ArPo-correction}.
\end{proof}

\subsection{The dualizing sheaf}
Let
$$U := \bigcup_{\operatorname{crk} F = 1} U_F \subset Y_\cA$$
be the union of the dense stratum and all strata of codimension 1.
Let $\omega_{Y_{\cA}}$ be the pushforward of the canonical sheaf $\omega_U$, which is
Weil divisorial and is a dualizing sheaf because $Y_\cA$ is normal and Cohen--Macaulay \cite[Theorems 8.0.4 and 9.0.9]{CLS}.
Let $$K_\cA:= -2\sum_{\operatorname{crk} F = 1}D_F.$$

\begin{proposition}\label{canonical}
  There is a canonical isomorphism 
  $\omega_{Y_\cA} \cong \det(L^*) \otimes \cO_{Y_{\cA}}(K_{\cA})$.
\end{proposition}

\begin{proof}
Let $\Omega\in\det(L^*)$ be a nonvanishing volume form on $L$, which is unique up to scale, and which we can regard as a rational volume form on $Y_\cA$.
Then
  \[
    \omega_{Y_\cA} \cong \det(L^*) \otimes \cO_{Y_\cA}(D),
  \]
  where $D$ is the divisor of poles of $\Omega$.
  It remains only to show that $D = K_\cA$.
  
  Since $D$ is preserved by the action of $L$, we have $D = \sum_F c_F D_F$ for some collection of integers $\{c_F\}$.
  Fix a particular corank 1 flat $F$.  By Lemma~\ref{neighborhood}, there exists and isomorphism $$U_F \cong L_F \times Y_{\cA^F} \cong 
  L_F \times \P^1,$$ where the 
  inclusion $L \subset U_F \cong L_F \times \P^1$ is a linear isomorphism
  onto $L_F \times \A^1$.  Therefore $c_F$ is equal to the order of the volume form $dz$ at $\infty\in\P^1$, which is equal to $-2$.
\end{proof}

\subsection{Equivariant minimal free resolution}\label{sec:resolution}
Before proving Theorem~\ref{resolution}, we explain why it is possible to choose a minimal free resolution of $\cO_{Y_{\cA}}$ that
is equivariant with respect to the action of $\Aut(\cA)$. Since $\Aut(\cA)$ 
is an extention of a finite group by a torus, this reduces to a statement about finite group equivariant multigraded minimal free resolutions
of modules over rings.
Let $$R = \bigoplus_{\mathbf{a} \in \mathbf{N}^k} 
R_{\mathbf{a}}$$ be a multigraded $\F$-algebra.  Let $\Gamma$ be a finite group acting on both $\N^k$ and $R$, with the property that $\gamma\cdot R_{\mathbf{a}}\subset R_{\gamma(\mathbf{a})}$
for all $\gamma\in\Gamma$ and $\mathbf{a} \in \mathbf{N}^k$.
Let $M$ be a finitely generated $\Gamma$-equivariant multigraded $R$-module.

\begin{lemma}\label{module res}
The module $M$ admits a minimal free resolution of the form
  \[
    \ldots \to Q_2 \otimes R \to Q_1 \otimes R \to Q_0 \otimes R \to M 
    \to 0,
  \]
  where each $Q_i$ is a $\Gamma$-representation and the maps are 
  $\Gamma$-equivariant.
\end{lemma}

\begin{proof}
Consider the surjection $M \twoheadrightarrow M \otimes_R \F$ of multigraded representations of $\Gamma$. 
By averaging, we can construct a $\Gamma$-equivariant section
$s_0:M \otimes_R \F \to M$. Let $Q_0 := M \otimes_R \F$, and consider the map \[
    d_0 := s_0 \otimes \operatorname{id}: Q_0 \otimes R \twoheadrightarrow M.
  \]
Let $K_0 := \ker d_0$, and let $s_1$ be a $\Gamma$-equivariant section 
  of the surjection $K_0 \twoheadrightarrow K_0 \otimes_R \F$. Now let 
  $Q_1 := K_0 \otimes_{R} \F$, and define $d_1 := s_1 \otimes 
  \operatorname{id}$. Continuing this process produces the desired 
  resolution.
\end{proof}

We will apply Lemma \ref{module res} to $$R = \bigotimes_{e\in E} \Sym(V_e\oplus \F)$$
with $k = |E| + c$, where $c = \dim \Aut(\cA)$ is equal to the number of connected components of the matroid $M_\cA$,
and $\Gamma$ is equal to the group of connected components of $\Aut(\cA)$.  This gives us the following result.

\begin{corollary}\label{emfr}
Any $\Aut(\cA)$-equivariant coherent sheaf on $\P_\cA$ admits a minimal free resolution by $\Aut(\cA)$-equivariant sheaves
with $\Aut(\cA)$-equivariant differentials.
\end{corollary}

\subsection{Identifying the syzygies}
For each element $e\in E$, define the divisor $D_0(e) := \sum_{e \not\in F} D_F$, so that 
$\pi_e^* \cO_{\overline{V_e}}(1) \cong \cO_{Y_\cA}(D_0(e))$. For each subset $S 
\subset E$, define $D_0(S) := \sum_{e \in S} D_0(e)$, so that 
$$\cO_{Y_\cA}(1_S) := \bigotimes_{e \in S} \pi_{e}^*\cO_{\overline{V_e}}(1) \cong 
\cO_{Y_\cA}(D_0(S)).$$ 

\begin{remark}\label{linearization}
Let $\Aut(\cA)_S\subset\Aut(\cA)$ be the stabilizer of the subset $S\subset E$.
The natural embedding of $\cO_{Y_\cA}(D_0(S))$ into the sheaf of 
rational functions on $Y_\cA$ induces an $\Aut(\cA)_S$-linearization of $\cO_{Y_\cA}(1_S)$.
In particular when $S=E$, it induces an $\Aut(\cA)$-linearization of $\cO_{Y_\cA}(1_E)$.
It also induces an $\Aut(\cA)$-linearization of the right-hand side of the equation in Theorem \ref{resolution}.
\end{remark}

Let $$\cdots\to \cQ_2\to\cQ_1\to\cQ_0\to\cO_{Y_\cA}\to 0$$
be an $\Aut(\cA)$-equivariant minimal free resolution of the structure 
sheaf $\cO_{Y_\cA}$ on $\P_\cA$, which exists by Corollary \ref{emfr}.  As in the statement of Theorem \ref{resolution}, let $M = M_{\cA^!}$.
It follows from \cite[Theorem 1.5]{ArBoo} that
\begin{equation*}\cQ_i \;\;\;\;\;\cong \bigoplus_{\operatorname{crk}_{M}(E\setminus S) = i}
Q_{S}\otimes \cO_{\PA}(-1_S),
\end{equation*}
where $Q_{S}$ is a representation of $\Aut(\cA)_S$.  

\begin{lemma}\label{ss}
For any subset $S\subset E$,
there is a canonical isomorphism
$$Q_S\otimes H^{|S|}( \P_{\cA_S}, \cO(-2_S))  \cong 
  H^{\operatorname{rk}_{M_{\!\cA}}\!(S)}(Y_{\cA}, \cO(-1_S)).$$
\end{lemma}

\begin{proof}
We begin by noting that the K\"unneth theorem provides an isomorphism
$$H^{|S|}( \P_{\cA_S}, \cO(-2_S))  \cong H^{|S|}( \P_{\cA}, \cO(-2_S)).$$  
Now consider the resolution
\[
0 \to \cQ_{\operatorname{rk} M}(-1_S) \to \cdots\to 
\cQ_2(-1_S)\to\cQ_1(-1_S)\to\cQ_0(-1_S)\to\cO_{Y_\cA}(-1_S)\to 0
\]
of $\cO_{Y_\cA}(-1_S)$.
Since all cohomology groups of $\cO_{\P^1}(-1)$ vanish, the K\"unneth theorem implies that the same is true of $\cO_{\P_\cA}(-1_S-1_T)$ unless $S=T$.
It follows that the cohomology of $\cQ_i(-1_S)$ vanishes unless $i=\operatorname{crk}_{M}(E\setminus S)$,
in which case it is isomorphic to
$$Q_S \otimes H^{|S|}(\P_\cA, \cO_{\P_\cA}(-2_S)).$$
If we regard $\cQ_\bullet(-1_S)$ as a complex of sheaves with $\cQ_i(-1_S)$ sitting in degree $-i$, this complex is quasi-isomorphic to $\cO_{Y_\cA}(-1_S)$, and we therefore have
$$Q_{S} \otimes H^{|S|}(\P_\cA, \cO_{\P_\cA}(-2_S))\cong \mathbb{H}^{|S|-\operatorname{crk}_{M}(E\setminus S)}(\P_\A, \cQ_\bullet(-1_S))
\cong H^{|S|-\operatorname{crk}_{M}(E\setminus S)}(Y_\cA, \cO(-1_S)).$$
The lemma now follows from the fact that $|S|-\operatorname{crk}_{M} = \rk_{M^!}(S) = \rk_{M_\cA}(S)$.
\end{proof}

Consider the projection map $\pi_S:\P_{\cA} \to \P_{\cA_S}$ along with its restriction $\pi_S:Y_{\cA} \to Y_{\cA_S}$. 

\begin{lemma}\label{resolve in stages}
The pullback
$$\pi_S^*:H^{\operatorname{rk}_{M_{\!\cA}}\!(S)}(Y_{\cA_S}, \cO(-1_S))\to H^{\operatorname{rk}_{M_{\!\cA}}\!(S)}(Y_{\cA}, \cO(-1_S))$$
is an isomorphism.
\end{lemma}

\begin{proof}
Choose $T \subset E\setminus S$ such that 
$|T| = \operatorname{crk}_{M_\cA}(S)$ and $S \cup T$ is a spanning set for $M_\cA$. 
Then $\pi_S: Y_{\cA} \to Y_{\cA_S}$ can be factored as a birational 
proper morphism $$\pi_{S\cup T}:Y_{\cA}\to Y_{\cA_{S \cup T}}$$ followed by a projection 
$$\rho_{S,T}:Y_{\cA_{S \cup T}} \cong Y_{\cA_S}\times (\P^1)^T \to Y_{\cA_S}.$$
Because Schubert varieties have rational 
singularities (Theorem \ref{rational}), the induced map
\[
  \pi_{S \cup T}^*:H^{\operatorname{rk}_{M_{\!\cA}}\!(S)}(Y_{\cA_{S \cup T}}, 
  \cO(-1_S)) \to H^{\operatorname{rk}_{M_{\!\cA}}\!(S)}(Y_{\cA_S}, 
  \cO(-1_S))
\]
is also an isomorphism.  Thus the composition
$\pi_S^* = \pi_{S \cup T}^*\circ \rho_{S\cup T}^*$ is an isomorphism.
\end{proof}

\begin{proof}[Proof of Theorem \ref{resolution}]
By Lemmas \ref{ss} and \ref{resolve in stages}, we have
\begin{equation}\label{so far}Q_S\otimes H^{|S|}( \P_{\cA_S}, \cO(-2_S))  \cong H^{\operatorname{rk}_{M_{\!\cA}}\!(S)}(Y_{\cA_S}, \cO(-1_S)).\end{equation}
By Proposition~\ref{canonical}, we have $$\omega_{\P_{\!\cA_S}} \cong \det(V_S^*)\otimes \cO(-2_S)
\and \omega_{Y_{\!\cA_S}} \cong \det(L_S^*)\otimes\cO(K_{\cA_S}) \cong  \det(L_S^*)\otimes\cO(-1_S)\otimes\cO(D_{-2}),$$
thus we can rewrite Equation \eqref{so far} as
$$
\det(V_S)\otimes Q_S\otimes H^{|S|}( \P_{\cA_S}, \omega_{\P_{\!\cA_S}})  \cong 
\det(L_S)\otimes H^{\operatorname{rk}_{M_{\!\cA}}\!(S)}(Y_{\cA_S}, \omega_{Y_{\!\cA_S}}\otimes \cO(-D_{-2})).
$$
By Serre duality, we have
\beq \det(V_S)\otimes Q_S &\cong& \det(V_S)\otimes Q_S\otimes H^{|S|}( \P_{\cA_S}, \omega_{\P_{\!\cA_S}})\\
&\cong& \det(L_S)\otimes  H^{\operatorname{rk}_{M_{\!\cA}}\!(S)}(Y_{\cA_S}, \omega_{Y_{\!\cA_S}}\otimes \cO(-D_{-2}))\\
&\cong& \det(L_S)\otimes H^0(Y_{\cA_S}, \cO(D_{-2}))^*.
\eeq
The theorem now follows from the isomorphism $\det(L_S)\cong \det(V_S)\otimes\det(L_S^\perp)$.
\end{proof}

\begin{proof}[Proof of Corollary~\ref{cohomology}]
Consider the resolution 
$$\cdots\to \cQ_2(1_E)\to\cQ_1(1_E)\to\cQ_0(1_E)\to\cO_{Y_\cA}(1_E)\to 0$$
obtained from Theorem~\ref{resolution} by tensoring with 
$\cO(1_E)$. By \cite{Brion}, all of the sheaves $\cQ_i(1_E)$ as well as 
$\cO_{Y_\cA}(1_E)$ have vanishing higher cohomology, so the hypercohomology spectral sequence
tells us that their sections form an exact sequence.
Finally, $H^0(\cO_{Y_\cA}(1_E)) \cong \cR_{+}(A)^*$ by Theorem~\ref{Schubert}.
\end{proof}

\subsection{The complete graph}
Let $\cA := \cA^!_{K_n}$ be the cographical linear space associated with the complete graph $K_n$.
Then $M = M_{\cA^!} = M_{K_n}$ is the braid matroid, whose flats of corank $i$ are indexed by partitions of the set $[n]$
into $i+1$ parts.  Given such a partition $\pi = (\pi_0,\ldots,\pi_i)$, the associated flat $F_\pi$ is the union of the complete
graphs on the parts of $\pi$, whereas the complement $S_\pi:= E\setminus F_\pi$ is the complete multipartite graph consisting of edges
that go between the blocks.  Thus we have
\beq
H^0(\cQ_i(1_E)) &\cong& \bigoplus_{\crk_M{E\setminus S} = i} 
\det(L_S^\perp) \otimes \cR_-(\cA_S) \otimes \bigotimes_{e \not\in 
S}(V_e^* \oplus \Q)\\
&\cong& \bigoplus_{\substack{\pi = (\pi_0,\ldots,\pi_i)\\ [n] = \pi_0\sqcup\cdots\sqcup\pi_i}}
\det(L_{S_\pi}^\perp) \otimes \cR_-(\cA_{S_\pi}) \otimes \bigotimes_{e \in 
F_\pi}(V_e^* \oplus \Q).
\eeq
Let $\mathfrak{S}_\pi \subseteq 
\mathfrak{S}_n$ be the 
subgroup of permutations that preserves the set partition $\pi$.
We have a homomorphism $\rho_\pi:\fS_\pi\to\fS_{i+1}$ that records the way in which an element $\fS_\pi$ permutes the blocks of $\pi$,
with kernel equal to the Young subgroup $\fS_{\pi_0}\times\cdots\times\fS_{\pi_i}$.
Each of the tensor factors above is a representation of $\fS_\pi$, and we will analyze each one individually.

Let $\Gamma_\pi$ be the graph obtained from $K_n$ by contracting the edges in $F_\pi$; this is a loopless graph with vertex set $\{0,\ldots,i\}$
and $|\pi_j| \cdot |\pi_k|$ edges between distinct vertices $j$ and $k$.
Then
$$\cA_{S_\pi} = (\cA_{K_n}^!)_{S_\pi} = (\cA_{K_n}^{F_\pi})^! \cong \cA_{\Gamma_\pi}^!.$$ 
This implies that
$$L_{S_\pi} = H_1(\Gamma_\pi;\Q)\subset C_1(\Gamma_\pi;\Q),$$
and therefore $$L_{S_\pi}^\perp = B^1(\Gamma_\pi;\Q) = \ker\left( C^1(\Gamma_\pi;\Q)\to H^1(\Gamma_\pi;\Q)\right).$$
The space of coboundaries of a graph is canonically isomorphic to that of its simplification, thus
$$L_{S_\pi}^\perp = B^1(\Gamma_\pi;\Q)\cong B^1(K_{i+1}; \Q).$$
This is isomorphic to the irreducible representation $V_{[i,1]}$ of $\fS_{i+1}$,
hence $$\det(L_{S_\pi}^\perp) \cong \wedge^i(V_{[i,1]}) \cong \operatorname{sign}_{\fS_{i+1}},$$
which we regard as a representation of $\fS_\pi$ via the homomorphism $\rho_\pi:\fS_\pi\to\fS_{i+1}$.
Since the multiplicative group $\Gm\subset\Aut(\cA)$ acts on $L_{S_\pi}^\perp$ with weight $1$, we will write
$$\det(L_{S_\pi}^\perp) \cong \operatorname{sign}_{\fS_{i+1}}[-i]$$
to indicate that it is sitting in degree $i$.

Simplifying a loopless graph by deleting parallel edges does not affect the internal zonotopal algebra of the cographical linear space; this can be seen in many different ways, for example
from \cite[Theorem 1.9]{CDBHP}.  Thus we have $\fS_\pi$-equivariant isomorphisms
$$\cR_-(\cA_{S_\pi}) \cong \cR_-(\cA_{\Gamma_\pi}^!)\cong \cR_-(\cA_{K_{i+1}}^!)\cong \OTbar(\cA_{K_{i+1}})^*,$$
where the last isomorphism follows from Theorem \ref{internal}.  Note that the reduced Orlik--Terao algebra is generated in degree 1,
and dualizing negates degrees.

\begin{remark}
The graded $\fS_{i+1}$-representation $\OTbar(\cA_{K_{i+1}})$ is studied in \cite{MPY}, where in particular one can find an explicit algorithm
to compute it recursively. Alternatively, $\cR_-(\cA_{K_{i+1}}^!)$ is isomorphic
in degree $-j$ 
to $H^{2j}(X(SU(2), K_{i+1});\Q)$ 
\cite{Berget-complex-reflection, CDBHP}, which is Whitehouse's lift of an 
Eulerian representation of $\fS_i$, tensored with 
$\operatorname{sign}_{\fS_{i+1}}$ \cite[Proposition 2]{ER19}.
\end{remark}

Finally, for any subgraph $\Gamma$ of $K_n$, let $$\Theta_\Gamma := \bigotimes_{e\in E(\Gamma)} V_e^*;$$
this is a (possibly nontrivial) 1-dimensional representation of $\Aut(\Gamma)$ sitting in degree $|E(\Gamma)|$.
Then $$\bigotimes_{e \in F_\pi}(V_e^* \oplus \Q) \cong 
\bigoplus_{\Gamma\in \cF_\pi} \Theta_\Gamma,$$
where $\cF_\pi$ is the set of all subgraphs of $F_\pi$.
Putting it all together, we have
$$H^0(\cQ_i(1_E)) \;\;\cong\;\; \bigoplus_{\substack{\pi = (\pi_0,\ldots,\pi_i)\\ [n] = \pi_0\sqcup\cdots\sqcup\pi_i}}
\operatorname{sign}_{\fS_{i+1}}[i]\;\otimes\;\OTbar(\cA_{K_{i+1}})^*
\;\otimes\; \bigoplus_{\Gamma\in \cF_\pi}\Theta_\Gamma,
$$
where the first two tensor factors are regarded as representations of $\fS_\pi$ via the homomorphism $\rho_\pi$.

\begin{corollary}\label{cobraid}
We have an equality
\beq \cR_+(\cA_{K_n}^!)^* &=& \sum_{i=0}^n (-1)^i \bigoplus_{\substack{\pi = (\pi_0,\ldots,\pi_i)\\ [n] = \pi_0\sqcup\cdots\sqcup\pi_i}}
\operatorname{sign}_{\fS_{i+1}}[i]\;\otimes\;\OTbar(\cA_{K_{i+1}})^*
\;\otimes\; \bigoplus_{\Gamma\in \cF_\pi}\Theta_\Gamma\\
&=& \bigoplus_\Gamma \Theta_\Gamma\;\otimes\bigoplus_{\pi:\Gamma\in \cF_\pi}(-1)^{|\pi|-1}
\operatorname{sign}_{\fS_{|\pi|}}[|\pi|-1]\;\otimes\;\OTbar(\cA_{K_{|\pi|}})^*
\eeq
of graded virtual representations of $\fS_n$.
\end{corollary}

\begin{remark}
  The dimension of the left hand side of Corollary~\ref{cobraid} is the 
  number of connected graphs on $[n]$, and the dimension of 
  $\OTbar(\cA_{K_{i+1}})$ is the M\"obius function value $\mu(\pi, 
  \hat{1})$ of the partition lattice, where $\hat{1}$ is the single block partition. In this case, Remark~\ref{spanning} says that 
  Corollary~\ref{cobraid} categorifies the formula 
  \[
    \text{number of connected graphs on $[n]$} = \sum_{\pi = (\pi_0, 
    \ldots, \pi_i)} 
    \mu(\pi, \hat{1}) \cdot 2^{{|\pi_0| \choose 2} + \cdots + {|\pi_i| 
    \choose 2}},
  \]
  which is obtained via M\"obius inversion from the formula 
  expressing the total number of graphs on $[n]$ as the sum over $\pi= (\pi_0, \ldots, \pi_i)$ of the number of 
  graphs with connected components $\pi$.
\end{remark}

\begin{example}
Consider the special case where $n=3$.  We will work with the first of the two expressions for $\cR_+(\cA_{K_3}^!)^*$ in Corollary \ref{cobraid}.
\begin{itemize}
\item When $i=0$, $\pi$ is the partition of $[3]$ into a single block, 
$\operatorname{sign}_{\fS_{1}}[0]\;\otimes\;\OTbar(\cA_{K_{1}})^*$ is trivial,
and $\cF_\pi$ consists of all simple graphs on the set $[3]$.  The empty graph contributes the trivial representation $V_{[3]}$ in degree 0.
The three graphs with one edge contribute the representation $V_{[2,1]}\oplus V_{[1,1,1]}$ in degree 1; the appearance of the sign representation
is due to the fact that, if $e$ is a single edge, then the stabilizer of $e$ acts nontrivially on the line $V_e^*$.  The three graphs with two edges
contribute the representation $V_{[2,1]}\oplus V_{[3]}$ in degree 2; the appearance of the trivial representation
is due to the fact that, if $e$ and $f$ is a pair of edges, then the stabilizer of the set $\{e,f\}$ acts trivially on the line $V_e^*\otimes V_f^*$.
Finally, the graph with all three edges contributes $\Theta_{K_3} \cong V_{[1,1,1]}$ in degree 3.
\item When $i=1$, $\pi$ ranges over the three different partitions of $[3]$ into two blocks.  For each such $\pi$, the homomorphism $\rho_\pi$
is trivial (the two blocks have different sizes and therefore cannot be permuted), so $\operatorname{sign}_{\fS_{2}}[-1]\;\otimes\;\OTbar(\cA_{K_{2}})^*\cong \Q[-1]$.  For each $\pi$, $\cF_\pi$
consists of two graphs, one of which is empty and one of which contains a single edge.  
The empty graph contributes the trivial representation of $\fS_\pi\cong S_2$ in degree 0; tensoring with $\Q[-1]$ and inducing to $\fS_3$ gives $V_{[3]}\oplus V_{[2,1]}$ in degree 1.  The graph with one edge
contributes the sign representation of $\fS_\pi\cong S_2$ in degree 1; tensoring with $\Q[-1]$ and inducing to $\fS_3$ gives $V_{[2,1]}\oplus V{[1,1,1]}$ in degree 2.
\item When $i=2$, $\pi$ is the partition of $[3]$ into three blocks, $\operatorname{sign}_{\fS_{2}}[2]$ is the sign representation in degree 2, and
$\OTbar(\cA_{K_3})^*$ is trivial in degree 0 and sign in degree $-1$.  The set $\cF_{\pi}$ contains only the empty graph, and $\Theta_{\emptyset}$ is trivial in degree zero.  Thus the total contribution when $i=2$ is $V_{[1,1,1]}\otimes V_{[1,1,1]}\cong V_{[3]}$ in degree 1 plus $V_{[1,1,1]}\otimes V_{[3]}\cong V_{[1,1,1]}$ in degree 2.
\end{itemize}
Taking the alternating sum, we find that $\cR_+(\cA_{K_3}^!)^*$ is isomorphic to the trivial representation of $\fS_3$ in degrees 0 and 2 plus the sign 
representation of $\fS_3$ in degrees 1 and 3.  This can confirmed by direct calculation, as the external zonotopal algebra of $\cA_{K_3}^!$ is generated by
a single element $\a$ in degree $-1$ that is not fixed by $\fS_3$, modulo the relation $\a^4=0$.

\end{example}

\bibliographystyle{amsalpha}
\bibliography{symplectic}

\end{document}